\documentclass{amsart}
\usepackage{amsmath}
\usepackage{amsthm}
\usepackage{amsfonts,mathrsfs}
\usepackage[flushleft]{threeparttable}
\usepackage{amssymb}
\usepackage{graphicx}
\usepackage{enumitem}
\usepackage{color}
\usepackage{verbatim}
\theoremstyle{plain}
\newtheorem{theorem}{Theorem}[section]
\newtheorem{lemma}[theorem]{Lemma}

\newtheorem{corollary}[theorem]{Corollary}
\newtheorem{proposition}[theorem]{Proposition}

\theoremstyle{definition}

\newcommand\Tstrut{\rule{0pt}{2.6ex}}         

\newtheoremstyle{TheoremNum}
	{\topsep}{\topsep}              
  {\itshape}                      
  {}                              
  {\bfseries}                     
  {.}                             
  { }                             
  {\thmname{#1}\thmnote{ \bfseries #3}}
\newtheorem{remark}{Remark}

\newcommand{\F}{\mathbb F}
\newcommand{\K}{\mathbb K}
\newcommand{\Z}{\mathbb Z}

\newcommand{\cD}{\mathcal D}

\newcommand{\cL}{\mathcal L}

\newcommand{\cV}{\mathcal V}
\newcommand{\cB}{\mathcal B}

\newcommand{\U}{\mathcal U}
\newcommand{\rS}{\mathscr S}

\newcommand{\Oval}{\mathcal O}
\newcommand{\bbS}{\mathbb S}
\newcommand{\bbP}{\mathbb P}

\newcommand{\Aut}{\mathrm{Aut}}

\newcommand{\PG}{\mathrm{PG}}

\newcommand{\Tr}{ \ensuremath{ \mathrm{Tr}}}
\newcommand{\autfix}[2]{\Gamma_{(#1,~#2)}}
\newcommand{\RN}[1]{%
  \textup{\uppercase\expandafter{\romannumeral#1}}%
}
\newcommand{\rn}[1]{%
  \textup{\lowercase\expandafter{\romannumeral#1}}%
}

 \def\zhou#1 {\fbox {\footnote {\ }}\ \footnotetext { From Yue: {\color{red}#1}}}

 \def\rocco#1 {\fbox {\footnote {\ }}\ \footnotetext { From Rocco: {\color{blue}#1}}}

\begin{document}
	\title{Hyperovals in Knuth's binary semifield planes}
	\author[N. Durante]{Nicola Durante}
	\address{Dipartimento di Mathematica e Applicazioni ``R. Caccioppoli", Universit\`{a} degli Studi di Napoli ``Federico \RN{2}", I-80126 Napoli, Italy}
	\email{ndurante@unina.it}
	\author[R. Trombetti]{Rocco Trombetti}
	\address{Dipartimento di Mathematica e Applicazioni ``R. Caccioppoli", Universit\`{a} degli Studi di Napoli ``Federico \RN{2}", I-80126 Napoli, Italy}
	\email{rtrombet@unina.it}
	\author[Y. Zhou]{Yue Zhou}
	\address{Dipartimento di Mathematica e Applicazioni ``R. Caccioppoli", Universit\`{a} degli Studi di Napoli ``Federico \RN{2}", I-80126 Napoli, Italy}
	\curraddr{Lehrstul f\"{u}r Diskrete Mathematik, Optimierung, und Operations Research, Universit\"{a}t Augsburg, 86135 Augsburg, Germany}
	\email{yue.zhou.ovgu@gmail.com}
	\date{\today}
	\maketitle
	
	\begin{abstract}
	In each of the three projective planes coordinatised by the Knuth's binary semifield $\K_n$ of order $2^n$ and two of its Knuth derivatives, we exhibit a new family of infinitely many translation hyperovals. In particular, when $n=5$, we also present complete lists of all translation hyperovals in them. The properties of some designs associated with these hyperovals are also studied.
	\end{abstract}

\section{Introduction}
A \emph{semifield} $\bbS$ is an algebraic structure satisfying all the axioms of a skewfield except (possibly) the associativity. A finite field is a trivial example of a semifield. Furthermore, if $\bbS$ does not necessarily have a multiplicative identity, then it is called a \emph{presemifield}. In particular, when the multiplication of $\bbS$ is commutative, we also say that $\bbS$ is \emph{commutative}.

For a presemifield $\bbS$, its additive group $(\bbS,+)$ is necessarily abelian \cite{knuth_finite_1965}. The multiplication of $\bbS$ is not necessarily commutative or associative. However, by Wedderburn's Theorem, in the finite case, associativity implies commutativity. Therefore, a non-associative finite commutative semifield is the closest algebraic structure to a finite field.

Geometrically speaking, there is a well-known correspondence, via coordinatisation, between (pre)semifields and projective planes of Lenz-Barlotti type V.1, see \cite{dembowski_finite_1997,hughes_projective_1973}. In \cite{albert_finite_1960}, Albert showed that two (pre)semifields coordinatise isomorphic planes if and only if they are \emph{isotopic}. Any presemifield can always be ``normalized" into a semifield under a certain isotopism, which implies that there is no essential difference between semifields and presemifields. An isotopism from a semifield to itself is called an \emph{autotopism}, and all autotopisms of a semifield form a group. We refer to \cite{lavrauw_semifields_2011} for a recent and comprehensive survey on semifields.

In \cite{knuth_class_1965}, Knuth presented a family of binary presemifields defined over $\F_{2^{mn}}$ where $n$ is odd. The multiplication in it is defined by
\[x*y := xy + \left(y\Tr_{2^{mn}/2^m}(x)+x\Tr_{2^{mn}/2^m}(y)\right)^2,\]
where $\Tr_{2^{mn}/2^m}(\cdot)$ is the trace function from $\F_{2^{mn}}$ to $\F_{2^m}$. It is clear that the multiplication $x*y$ is commutative, i.e.\ $x*y=y*x$. In this paper, we concentrate on the Knuth's binary presemifields with $m=1$ and denote them by $\K_{n}$. For convenience, we write $\Tr_{2^n/2}(x)$ as $\Tr(x)$.

Let $\Pi$ be a finite projective plane of order $q$. An \emph{oval} is a set of $q+1$ points, no three of which are collinear. Due to Segre's theorem \cite{segre_ovals_1955}, every oval in $\PG(2,q)$ with $q$ odd is equivalent to a nondegenerate conic in the plane. When $q$ is even, every oval can be uniquely extended to a set of $q+2$ points, no three of which are still collinear. We call this set of $q+2$ points a \emph{hyperoval}. Equivalently, we can also say a hyperoval is a \emph{maximal arc} of degree $2$. Dually, a \emph{line hyperoval} is a set of $q+2$ lines, no three of which are concurrent.

From the definition of hyperovals, it is easy to show that each line in $\Pi$ meets a hyperoval at either $0$ or $2$ points, and the line is called \emph{exterior} or \emph{secant}, respectively.  For two different hyperovals $\Oval_1$ and $\Oval_2$ in $\Pi$, we say that they are \emph{equivalent} if there is a collineation $\varphi$ of the plane such that $\varphi(\Oval_1)=\Oval_2$.

When $q$ is even, a conic together with its nucleus in $\PG(2,q)$ is a \emph{regular hyperoval} (also known as a complete conic). Different from the $q$ odd case, for $q$ even, there are several inequivalent hyperovals in $\PG(2,q)$ besides the regular hyperovals. Finding these hyperovals has been a hard work of almost 40 years and a complete classification of hyperovals in $\PG(2,q)$ seems elusive. We refer to \cite{cherowitzo_hyperovals_1996} and \cite{caullery_classification_2015,hernando_proof_2012} for a list of known hyperovals in $\PG(2,q)$ and recent classification results, respectively.

Let $\Pi$ be a translation plane of even order $q$ with a translation line $l_\infty$. A hyperoval $\Oval$ is called a \emph{translation hyperoval}, if
\begin{itemize}
	\item the line $l_\infty$ is secant to $\Oval$,
	\item there exists a subgroup of order $q$ in the translation group acting regularly on the $q$ affine points of $\Oval$. 
\end{itemize}

The set $\l_\infty \cap \Oval$ is also known as the {\it carrier set} of the hyperoval. When $\Pi$ is $\PG(2,q)$ and $q=2^n$, without loss of generality, we can always assume that $(\infty)$, $(0)$ and $(0,0)$ are points of the hyperoval $\Oval$ and the affine points of $\Oval$ can be written as 
\[\{(x,f(x)) : x\in \F_q  \}.\]
In \cite{hirschfeld_ovals_1975,payne_complete_1971}, it has been shown that a function $f$ defines a translation hyperoval in $\PG(2,q)$ if and only if $f(x)=x^{2^k}$ where $\gcd(n,k)=1$. It implies a complete classification of translation hyperovals in desarguesian planes.

There are a few known hyperovals in non-desarguesian planes of even orders, which were discovered approximately 20 years ago. In generalized Andr\'{e} planes, there are translation hyperovals inherited from those in desarguesian planes; see \cite{denniston_non-desarguesian_1979,jha_ubiquity_1992}. In a Hall plane of order $q^2$, the affine points of certain conics in $\PG(2,q^2)$ can also be extended into hyperovals; see \cite{korchmaros_inherited_1986,okeefe_hyperovals_1992}. In finite Figueroa planes of even order, there are hyperovals inherited from regular ones in the associated desarguesian planes; see \cite{de_resmini_hyperovals_1998}. A complete list of hyperovals in all projective planes of order $16$ is presented in \cite{penttila_hyperovals_1996}, in which it is also shown that there do exist projective planes containing no hyperovals. 

In this paper, we present three new families of translation hyperovals; one in the semifield plane $\Pi(\K_n)$ coordinated by the Knuth's binary presemifield $\K_n=(\F_{2^n},+,\ast)$, one in $\Pi(\K^t_n)$ and another one in $\Pi(\K_n^{td})$, where $t$ and $d$ are the {\it transpose} and the {\it dual} operation on $\K_n$, respectively. In particular, in each relevant plane we give a complete list of translation hyperovals, when $n=5$.

Throughout this paper, we denote the points and the lines of a semifield plane $\Pi(\bbS)$ in the following way: The set of affine points in $\Pi(\bbS)$ is $\{(a,b):  a, b\in \F_q\}$. There are another $q+1$ points on the line $l_\infty$ at infinity, and we denote them by $(a)$, where $a\in \F_q\cup \{\infty\}$. Symmetrically, all the $q+1$ lines through $(\infty)$ are denoted by $l_a$, where $a\in \F_q\cup \{\infty \}$. The affine points on $l_a$ are $(a,y)$ for all $y\in \F_q$. For any $a,b\in \F_q$, the line $l_{a,b}$ is defined by
\[l_{a,b} :=\{(x,y): y=x\star a + b\} \cup \{(a)\}, \]
where $\star$ is the multiplication in $\bbS$.

The rest of this paper is organized as follows: In Section \ref{se:commutative}, we first classify the translation hyperovals in semifield planes into two types (a) and (b). Then we concentrate on hyperovals in $\Pi(\K_n)$, and we present a new family of translation hyperovals of type (b) in it. Moreover, we give a complete list of translation hyperovals in $\Pi(\K_5)$. In Section \ref{se:symplectic}, instead of studying translation hyperovals in $\Pi(\K_n^t)$, we equivalently look at a certain type of line hyperovals in its dual plane $\Pi(\K_n^{td})$. A new family of line hyperovals in $\Pi(\K_n^{td})$ and a complete list of them in $\Pi(\K_5^{td})$ are both presented. In Section \ref{se:symplectic_hyperoval}, we also get a new family of translation hyperovals in $\Pi(\K_n^{td})$ and a complete list of them in $\Pi(\K_5^{td})$. Finally in Section \ref{se:designs}, we turn to the hyperovals obtained by applying collineations on the translation hyperovals in semifield planes 
 and the sizes of pairwise intersections of them. We prove that, for each hyperoval of type (a) in a semifield plane, there is a symmetric design as well as two associated difference sets with the same parameters $\left(q^2, {q^2}/{2} + {q}/{2}, {q^2}/{4} + {q}/{2}\right)$, which are defined in two different abelian groups.



\section{Hyperovals in $\Pi(\K_n)$}\label{se:commutative}

In \cite[Corollary 2]{jha_oval_1996}, it was proven that in every plane $\Pi(\bbP)$ coordinatised by a commutative (pre)semifield $(\bbP,+,\star)$ of even order there is a translation oval, whose set of affine points is 
\begin{equation}\label{set:standardovfal}
\{(x,x\star x): x\in \F_q \}.
\end{equation} 
In the same article the authors refer to such an oval as the {\it standard oval}. Naturally, one can always complete the standard oval into a hyperoval $\Oval_s$; for instance, by adding $\{(0), (\infty)\}$ as a carrier set. In  \cite{de_resmini_arcs_2002}, de Resmini, Ghinelli and Jungnickel used relative difference set to prove that in every commutative semifield plane, there is an oval, which is equivalent to the oval defined by \eqref{set:standardovfal}.

When $\bbP$ is a finite field, clearly $\Oval_s$ is the hyperoval comprising the points of a conic and its nucleus in $\PG(2,q)$, i.e., a regular hyperoval. 

As the presemifield $\K_n$ is commutative, if we replace $\star$ by its multiplication $*$ in \eqref{set:standardovfal},  we get a hyperoval $\Oval_s$ in $\Pi(\K_n)$.

In this section, we proceed to show new hyperovals in $\Pi(\K_n)$. To consider the equivalence between different hyperovals, we need to know the automorphism group of the projective plane $\Pi(\K_n)$. As Hughes and Piper showed in \cite[Lemma 8.4 and Theorem 8.6]{hughes_projective_1973}, the automorphism group of a semifield plane can be decomposed in the following way:
\begin{theorem}\label{th:automorphism_semifieldplane}
    Let $\Pi(\bbS)$ be a semifield plane defined by a semifield $\bbS$. Let $\autfix{l_\infty}{l_\infty}$ be its translation group and $\autfix{(\infty)}{l_0}$ the group of \emph{shears} (automorphisms fixing $(\infty)$ and $l_0$). Let $\Aut(\bbS)$ and $\Aut(\Pi(\bbS))$ denote the autotopism group of $\bbS$ and the automorphism group of $\Pi(\bbS)$, respectively. Then we have
    \begin{itemize}
        \item $\Sigma := \autfix{l_\infty}{l_\infty}\rtimes \autfix{(\infty)}{l_0}$ is a subgroup of $\Aut(\Pi(\bbS))$;
        \item $\Aut(\Pi(\bbS))\cong\Sigma \rtimes \Aut(\bbS)$.
    \end{itemize}
\end{theorem}

The autotopism groups of Knuth's binary semifields are determined in \cite{knuth_class_1965}. For $\K_n$, the corresponding result is as follows: 
\begin{lemma}\label{lm:autotopism_knuth}
	The autotopism group $\Aut(\K_n)$ of $\K_n$ is cyclic of order $n$. Moreover, one of its generator $\sigma$ is defined by
	\begin{align*}
		(x,y) &\mapsto (x^2,y^2),\quad \text{for all }x,y\in \F_{2^n},\\
		(z) &\mapsto(z^2), \quad \text{for each }z\in \F_{2^n},\\
		(\infty) & \mapsto (\infty). 
	\end{align*} 
\end{lemma}

By using Theorem \ref{th:automorphism_semifieldplane} and Lemma \ref{lm:autotopism_knuth}, we can prove the following result.
\begin{proposition}\label{prop:classification}
Let $\Pi(\bbS)$ be a semifield plane of even order $q$. Let $\Oval$ be a translation hyperovals of $\Pi(\bbS)$; then $\Oval$ can be written, up to the action of $\Aut(\Pi(\bbS))$, in one of the following forms:
\begin{enumerate}[label=(\alph*)]
	\item	$\Oval := \{(x, f_a(x)): x\in \F_q\} \cup \{(0), (\infty)\}$, where $f_a$ is an additive permutation on $\F_q$ and $f_a(0)=0$;
	\item  $\Oval := \{(f_b(y),y): y\in \F_q\} \cup \{(0), (\alpha)\}$, where $\alpha\in \F_q^*$, $f_b$ is a two-to-one additive map on $\F_q$ and $f_b(0)=0$.
\end{enumerate}
\end{proposition}
\begin{proof}
By definition, the autotopism group of a semifield $\bbS$ fixes $(\infty)$, $(0)$ and $(0,0)$ in $\Pi(\bbS)$. Together with Theorem \ref{th:automorphism_semifieldplane}, we see that the point $(\infty)$ is fixed under the automorphism group $\Aut(\Pi(\bbS))$ of $\Pi(\bbS)$. As a consequence we have that, up to the action of the automorphism group of the plane, translation hyperovals in $\Pi(\bbS)$ either contain the point $(\infty)$ or do not contain $(\infty)$. As the translation group is transitive on the affine points of $\Pi(\bbS)$, we can always assume that $(0,0)$ is in the translation hyperovals. Moreover, we can also assume that $(0)$ is in them, because of the transitivity of the shear group on the points in $l_\infty \setminus {(\infty)}$. This leads to shapes $(a)$ and $(b)$ where both $f_a$ and $f_b$ are additive over $\F_q$. In another word, they can be both expressed as linearized polynomials over $\F_q$. Also, according to the definition of hyperovals, $f_a$ must be a permutation on $\F_q$ and $f_b(y)=0$ has exactly two solutions in $\F_q$.
\end{proof}

It is clear that the standard hyperoval $\Oval_s$ is of type (a). In the Knuth's binary presemifield plane $\Pi(\K_5)$, we find $5$ inequivalent hyperovals of type (a) by using the computer algebra program MAGMA \cite{Magma}; see Table \ref{ta:hyperoval_a}. It is a natural question to ask whether any member of these sporadic examples, except for $f_a(x)=x^2$, can be generalized into an infinite family of hyperovals in $\Pi(\K_n)$. At least up to $n=11$, our MAGMA program shows that there is no other linearized polynomials with coefficients in $\{0,1\}$ defining new hyperovals in $\Pi(\K_n)$.

\begin{table}[h!]
\caption{Functions defining translation hyperovals of type (a) in $\Pi(\K_5) $}
\label{ta:hyperoval_a}
\centering{
	\begin{tabular}{|c|l|}    
	\hline No. & Functions\\
	\hline 1 &	$x^2$\\
	 2 &   $x^4 + x^2 + x$\\
	 3 &   $x^8$\\
	 4 &   $x^8 + x^4 + x^2$\\
	 5 &    $\omega^{7}x^{16} + \omega^{14}x^8 + \omega^{30}x^4 + \omega^{4}x^2 + \omega^{12}x$\\
	\hline
	\multicolumn{2}{|c|}{\small Here the primitive element $\omega$ satisfies $\omega^5 + \omega^2 + 1=0$.}	\Tstrut\\
	\hline 
	\end{tabular}}
\end{table}

Next, we turn to hyperovals of type (b) in $\Pi(\K_n)$. To get a new family of hyperovals, we need some results about Dickson's polynomials. The following lemma can be found in \cite[Proposition 6]{dillon_new_2004}; for more results concerning Dickson's polynomials, we refer to \cite{lidl_dickson_1993}. 
\begin{lemma}\label{lm:Dickson_polynomial}
	Let $n$ be an odd positive integer and $q:= 2^n$. Let $D_3$ be a map from $\F_q$ to itself given by the Dickson polynomial
	\[D_3(X)=X^3+X\in \F_q[X].\]
	For each $t\in\F_q$,
	\[ \#\{x\in \F_q : D_3(x)=t\} = 
	\left\{ \begin{array}{ll}
		1, & \hbox{if  $\Tr(t^{-1})=0$ and $t\neq 0$, $1$;}\\
		2, & \hbox{if $t=0$ or $t=1$;}\\
		3, & \hbox{if $\Tr(t^{-1})=1$ and $t\neq 0$, $1$.}
		  \end{array}
		\right.
	\]
\end{lemma}

\begin{theorem}\label{th:hyperoval_g1g2}
	Let $n$ be an odd positive integer and $q:=2^n$. Let $g(y)=y^2+y$ for $y\in\F_q$. In $\Pi(\K_n)$, the set
	\[\Oval_{g} := \{(g(y),y): y \in \F_{q}\} \cup \{(0), (1)\}\]
	is a translation hyperoval. Furthermore, the hyperoval $\Oval_{g}$ is inequivalent to the standard hyperoval $\Oval_s$.
\end{theorem}
\begin{proof}
	By definition, we need to show that each line of $\Pi(\K_n)$  intersects $\Oval_g$ at exactly zero or two points. Clearly it holds for the line $l_\infty$. For arbitrary $a\in \F_{q}$, the cardinality of $l_a\cap \Oval_g$ is $\#\{y: g(y)=a\}$, which equals $0$ or $2$.
	
	For any $a, b\in \F_q$, there are $\#\{y: y=a*g(y)+b\}$ points in $l_{a,b}\cap \Oval_g$. Noting that $g(y)$ is additive. We only have to show that
	\begin{itemize}
		\item  for every $a\in \F_{q}\setminus \{0,1\}$, there are exactly two $y\in \F_q$ such that $y=a*g(y)$;
		\item $y=1*g(y)$ only if $y=0$.
	\end{itemize}
	
	Expanding $y+a*g(y)$, we have
	\begin{align*}
		y+a*g(y)  &= y+a(y^2+y) + \left((y^2+y)\Tr(a) + (a)\Tr(y^2+y)\right)^2\\
			&= y + a(y^2+y) + (y^4+y^2)\Tr(a)\\
			&= \Tr(a)y^4 + (\Tr(a)+a)y^2 + (a+1)y.
	\end{align*} 
	
	When $\Tr(a)=0$ and $a\neq 0$, $y+a*g(y)=0$ becomes
	\[a y^2 + (a+1)y=0,\]
	which has exactly two roots $0$ and $1+\frac{1}{a}$.
	
	When $\Tr(a)=1$, $y+a*g(y)=0$ becomes
	\begin{equation}\label{eq:oval_g_1_y^4}
		y^4 + (a+1)y^2 + (a+1)y=0.
	\end{equation}
	If $a=1$, it is clear that \eqref{eq:oval_g_1_y^4} holds if and only if $y=0$. Now suppose that $a\neq1$. Diving \eqref{eq:oval_g_1_y^4} by $y$, we get
	\begin{equation}\label{eq:oval_g_1_y^3}
		y^3 + (a+1)y + (a+1)=0.
	\end{equation}
	Let $c= (a+1)^{2^{n-1}}$. As $n$ is odd and $\Tr(a)=1$, we know that $\Tr(c)=0$. Replacing $y$ by $cy$ in \eqref{eq:oval_g_1_y^3} and diving it by $c^3$, we obtain
	\[y^3 + y + \frac{1}{c}=0.\]
	From Lemma \ref{lm:Dickson_polynomial}, we know that \eqref{eq:oval_g_1_y^3} has a unique root because $\Tr(c)=0$ and $c\neq 0$. Therefore \eqref{eq:oval_g_1_y^4} has exactly two roots.
	
	According to Proposition \ref{th:automorphism_semifieldplane} and Lemma \ref{lm:autotopism_knuth}, all collineation of $\Pi(\K_n)$ fix the point $(\infty)$. Hence, from Proposition \ref{prop:classification} $\Oval_{g}$ is inequivalent to $\Oval_s$. 
\end{proof}

Similarly we can show that, for $g'(y)=y^{2^{n-1}} + y^{2^{n-2}}$, the set
	\[\Oval_{g'} := \{(g'(y),y): y \in \F_{q}\} \cup \{(0), (1)\}\]
is also a hyperoval in $\Pi(\K_n)$. However, $\Oval_{g}$ and $\Oval_{g'}$ are equivalent: 
We take the non-identity shear $\theta$ which fixes $\{(0) , (1)\}$ and is defined by
\begin{align*}
	(x,y) &\mapsto (x, y+1*x), \quad\text{for all }x,y\in \F_q,\\
	(z) &\mapsto (z+1),\quad\text{for every }z\in \F_q,\\
	(\infty) &\mapsto (\infty).
\end{align*}
Thus $\theta(y^2+y,y)= (y^2+y, y+ 1*(y^2+y))= (y^2+y, y^4)$. Replace $y^4$ by $z$, we get $(z^{2^{n-1}} + z^{2^{n-2}}, z)$,  which implies that $\theta(\Oval_{g})=\Oval_{g'}$.

Again by using MAGMA program, we find that there are totally $12$ inequivalent translation hyperovals of type (b) in $\Pi(\K_5)$. The corresponding functions $f_b$ and the points $(\alpha)$ are listed in Table \ref{ta:hyperoval_b}.

\begin{table}[h!]
\caption{Functions defining translation hyperovals of type (b) in $\Pi(\K_5)$}
\label{ta:hyperoval_b}
\centering{
	\begin{tabular}{|c|c|l|}
	\hline No.& $\alpha$ &  {Functions}  \\
	\hline
	$1$ & $1$ & $y^2 + y$, \Tstrut\\
	$2$	&  & $y^4 + y$  \Tstrut\\ \hline
	$3$ & $\omega$ & $\omega^{21} y^{16} + \omega^3 y^8 + \omega^{11} y^4 + \omega^{23} y^2 + \omega^{25} y$  
	\Tstrut \\\hline 
	$4$ & $\omega^3$ & $\omega^{13} y^{16} + \omega^{17} y^8 + \omega^4 y^4 + \omega^{16} y^2 + \omega^{14} y$, \Tstrut \\
	$5$ &	       & $\omega^{13} y^{16} + \omega^{25} y^8 + \omega^{17} y^4 + \omega^{14} y^2 + \omega^{12} y$, \Tstrut \\
	$6$ &		   & $\omega^{16} y^{16} + \omega^{16} y^8 + \omega^{17} y^4 + \omega^8 y^2 + \omega^6 y$     \\\hline
	$7$ & $\omega^7$ & $\omega^{26} y^{16} + \omega^{27} y^8 + \omega^{18} y^4 + \omega^{27} y^2 + \omega^{27} y$,  \Tstrut \\ 
	$8$ &		   & $\omega^{17} y^{16} + \omega^{18} y^8 + \omega^8 y^4 + \omega^{29} y^2 + \omega y$, \Tstrut \\ \hline
	$9$ & $\omega^{11}$ & $\omega^{20} y^{16} + \omega^{25} y^8 + \omega^{29} y^4 + \omega^{26} y^2 + \omega^5 y$,  \Tstrut\\
	$10$ &		    & $\omega^{19} y^{16} + \omega^2 y^8 + \omega^{28} y^4 + \omega^{27} y^2 + \omega^7 y$, \Tstrut\\
	$11$ &		   & $\omega^{11} y^{16} + \omega^{24} y^8 + \omega^{17} y^4 + \omega^{10} y^2 + \omega^4 y$ \Tstrut\\ \hline

	$12$ & $\omega^{15}$ & $\omega^{29} y^{16} + \omega^{21} y^8 + \omega^5 y^4 + \omega^{24} y^2 + \omega^{16} y$     \Tstrut\\ \hline
		\multicolumn{3}{|c|}{\small Here the primitive element $\omega$ satisfies $\omega^5 + \omega^2 + 1=0$.}	\Tstrut\\\hline 
	\end{tabular}
}
\end{table}

Combining the results in Table \ref{ta:hyperoval_a} and Table \ref{ta:hyperoval_b}, we have a complete classification of translation hyperovals in $\Pi(\K_5)$.

\begin{remark}
Let $(\bbS, +, \star )$ be a (pre)semifield of order $q$. Assume that $\Oval$ is an (hyper)oval in $\Pi(\bbS)$. Immediately we can get a line (hyper)oval in the dual plane. In particular, when $\star$ is commutative, we get a line (hyper)oval in $\Pi(\bbS)$ itself. For instance, the line hyperoval derived from the standard hyperoval $\Oval_s$ is
\[\{y=x\star a + a \star a: a\in \F_q \} \cup \{l_0, l_\infty\}. \]
Similarly, we can also get a family of line hyperovals in $\Pi(\K_n)$ from the hyperovals in Theorem \ref{th:hyperoval_g1g2} and more examples in $\Pi(\K_5)$ from the hyperovals listed in Table \ref{ta:hyperoval_a} and Table \ref{ta:hyperoval_b}.
\end{remark}

\section{Line hyperovals in Symplectic planes}\label{se:symplectic}
Let $(\bbP,+,\star)$ be a (pre)semifield of order $q$. By definition, for any $a\in \bbP$, the map $R_a$ given by $x\mapsto x \star a$ is additive on $\bbP$. Hence $R_a$ must be linear over a certain finite field $\F_r$ which is a subfield of  $\F_q$, and we can represent $R_a$ by a square matrix over $\F_r$ which is nonsingular if and only if $a\neq 0$.  Moreover, the set $\{R_a : a\in \bbP\}$ is also closed under addition. It is not difficult to show that the set $\{R^t_a : a\in \bbP\}$ where here $R^t_a$ stays for the transpose matrix of $R_a$, defines another (pre)semifield, which we call the \emph{transpose} of $\bbP$ and denote by $\bbP^t$. 

Another obvious (pre)semifield derived from $\bbP$ is the opposite algebra of $\bbP$, which coordinatises the dual plane of $\Pi(\bbP)$. We call it the \emph{dual} of $\bbP$ and denote it by $\bbP^d$. When $\bbP$ is commutative, it is clear that $\bbP=\bbP^d$.

In \cite{knuth_class_1965}, Knuth investigated the associated cubical array of a (pre)semifield $\bbP$ and obtained a chain of six ones, which are
\[\bbP, \bbP^d, \bbP^t, \bbP^{dt}, \bbP^{td}, \bbP^{dtd}. \]
We also call them the \emph{Knuth derivatives} of $\bbP$. 
In particular, when $\bbP$ is commutative, there are only three of them $\{\bbP, \bbP^t, \bbP^{td}\}$, and $\bbP^{td}$ is a symplectic semifield. 

A (pre)semifield $(\bbS, +, \star)$ is called \emph{symplectic} if \[ \Tr_{q/p}(x(y\star z)) = \Tr_{q/p}(y(x \star z)), \] for all $x,y,z\in \bbS$, where $q=\# \bbS$ and $q$ is a power of a prime $p$; see \cite{kantor_commutative_2003} for more details.

For the Knuth's binary presemifield $\K_n$, up to isotopism, the multiplication in the associated symplectic presemifield $\K_n^{td}$ is defined by
\[x\circ y:= xy + \Tr(x) \sqrt{y} + \Tr(x^2y),\]
for $y,z\in \F_q$, where $q=2^n$. It is worth noting that the autotopism group of $\K_n^{td}$ is the same as that of $\K_n$.

Next, let us explain a bit how to derive $\K_n^{td}$ from $\K_n$ in a different way. This approach is first applied by Kantor in \cite{kantor_commutative_2003}, and we use it later to derive line hyperovals in $\Pi(\K_n^{td})$ from the hyperovals of type (a) in $\Pi(\K_n)$.

For $a,b,c,d\in \F_q$, we define a nondegenerate alternating bilinear form by
\[\langle(a,b), (c,d)\rangle:= \Tr(ad+bc).\]  
It is straightforward to prove that 
\[\langle (x, x\circ z), (y, y\circ z) \rangle= \Tr(y(x \circ z)+ x (y\circ z)) = 0,\]
for all $x,y,z\in \F_q$, which is a verification of the symplecticity of $\K_n^{td}$.

Let $B$ be a bivariate polynomial given by 
\[B(X,Y)=\sum_{0\le i\le j\le n-1}a_{ij}X^{2^i}Y^{2^j}\in \F_q[X,Y].\] 
Clearly, for a given element $m\in\F_q$, the polynomial $B(X,m)$ and $B(m,X)$ are both linearized polynomials. Furthermore, we can always find another unique bivariate polynomial $\bar{B}(X,Y)$ such that
\begin{itemize}
	\item $\langle (x, B(x,m)), (y, \bar{B}(m,y)) \rangle =0$, for all $x,y,m\in \F_q$;
	\item for each $m\in \F_q$, the kernels of the two maps given by $x \mapsto B(x,m)$ and $y\mapsto \bar{B}(m,y)$ are of the same size.
\end{itemize}
This statement can be shown in the following way: If we choose a trace-orthogonal basis $\mathscr{B}$ of $\F_q$ over $\F_2$ and consider $a,b,c,d\in\F_q$ as elements in $\F_{2}^n$, then we can write $\langle(a,b), (c,d)\rangle= a \cdot d + b\cdot c$ in terms of the usual dot product of vectors. Let $M$ and $N$ be two $n\times n$ matrices over $\F_2$ corresponding to $B(x,m)$ and $\bar{B}(m,y)$ respectively. The $n$-dimensional subspaces $\{(x,xM) : x\in \F_{2}^n\}$ and $\{(x,xN) : x\in \F_{2}^n\}$ are orthogonal if and only if $M=N^T$, since $\langle (x,xM), (y, yN)\rangle = xN^Ty^t + xMy^t=x(N^T+M)y^t$. On the other hand, the additivity of $\bar{B}(\cdot, y)$ for given $y$ is directly derived from the additivity of $B(x, \cdot)$.

As an application, we can get the symplectic presemifield $\K_n^{td}$ from $\K_n$ by showing that 
\begin{equation}\label{eq:commutative_symplectic_0}
	\langle (x, x*z), (y, z\circ y) \rangle  = \Tr(x (z\circ y) + y (x*z) )=0,
\end{equation}
for all $x,y,z\in \F_q$. 


Let $L(X)=\sum_{i=0}^{n-1}a_i X^{2^i}$ be a linearized polynomial over $\F_{q}$. The polynomial
\[\bar{L}(X) =  \sum_{i=0}^{n-1}a_{i}^{2^{n-i}} X^{2^{n-i}}\]
is the \emph{adjoint} of $L$ with respect to the symmetric bilinear form $(x,y) \mapsto \Tr(xy)$. Hence, we have $\langle (x, L(x)), (y, \bar{L}(y)) \rangle  = 0$ for all $x,y\in \F_q$. Together with \eqref{eq:commutative_symplectic_0}, we obtain
\begin{equation}\label{eq:showing_line_hyperoval_type_a}
	\langle (x, x*z + L(x)), (y, z\circ y + \bar{L}(y)) \rangle  = 0,
\end{equation}
for all $x,y,z\in \F_q$. It implies that the kernel of $z\circ y + \bar{L}(y)$ is of the same size as that of $x*z + L(x)$. In particular, when $L(x)$ defines a hyperoval of type (a) in $\Pi(\K_n)$, we get an associated line hyperoval in $\Pi^{td}(\K_n)$; and vice versa. Hence, we obtain the following result.
\begin{theorem}\label{th:classical_line_hyperovals}
	Let $L$ be a linearized polynomial on $\F_q$. Then $L$ defines a hyperoval of type (a) in $\Pi(\K_n)$ if and only if the set of lines defined by
	\[\cL_c :=\{l_{m,\bar{L}(m)} : m\in \F_q \}\cup 
			\{l_0,{l}_\infty\},\] where $l_{m,\bar{L}(m)}=\{(x,y): y=x\circ m + \bar{L}(m)\} \cup \{(m)\}$
	is a line hyperoval in $\Pi^{td}(\K_n)$.
\end{theorem}

By applying Theorem \ref{th:classical_line_hyperovals} onto the standard translation hyperoval and onto those contained in Table \ref{ta:hyperoval_a}, we get line hyperovals. 
\begin{corollary}\label{coro:classical_symplectic}
	Let $q=2^n$. In the symplectic presemifield planes $\Pi(\K_n^{td})$.
	\begin{itemize}
		\item	The set of lines defined by
			\[\cL_c :=\{l_{m^2,m} : m\in \F_q \}\cup 
				\{l_0, {l}_\infty\}\]
			is a line hyperoval.
		\item In particular, when $n=5$, there are exactly $5$ line hyperovals,  each of which contains the lines $l_0$ and $l_\infty$.
	\end{itemize}
\end{corollary}

As shown by Maschietti in \cite{maschietti_symplectic_2003,maschietti_symmetric_2009}, for any symplectic quasifield (not necessarily semifield), there is always a line hyperoval containing the lines ${l}_\infty$ and $x=0$. In fact, the line hyperoval $\cL_c$ obtained in Corollary \ref{coro:classical_symplectic} can be also proved by using  Maschietti's approach.

Equivalently, we may also consider the line hyperovals in $\Pi(\K_n^{td})$ as the hyperovals in the dual plane $\Pi(\K_n^{t})$ of $\Pi(\K_n^{td})$. Hence, the line hyperovals presented in Corollary \ref{coro:classical_symplectic} correspond to translation hyperovals containing the points $(0)$ and $(\infty)$ in $\Pi(\K_n^t)$.

In $\Pi(\K_n)$, for every translation hyperoval $\Oval$, there is a subgroup of order $q$ of the translation group acting regularly on the affine points of $\Oval$. Dually, for every line hyperovals  $\cL$ from Theorem \ref{th:classical_line_hyperovals}, there is a subgroup of order $q$ of the collineation group $\Gamma_{(\infty), (\infty)}$ on $\Pi(\K_n^{td})$ acting regularly on the $q$ lines in $\cL \setminus \{l_0, {l}_\infty\}$. Here $\Gamma_{(\infty), (\infty)}$ is the group of all collineations fixing $(\infty)$ and all lines through $(\infty)$. It corresponds to the translation group on the dual plane $\Pi(\K_n^{t})$ of $\Pi(\K_n^{td})$. 

As in Section \ref{se:commutative}, up to the action of this group, we can classify the line hyperovals into two types: A line hyperoval of type (a) contains the lines $l_\infty$ and $l_0$; one of type (b) contains the lines $l_0$ and $l_\alpha$ for certain $\alpha\in \F_q^*$. Precisely, 

\begin{enumerate}
	\item[(a)]	$\cL := \{l_{x, L_a(x)}: x\in \F_q\} \cup \{l_0, l_{\infty}\}$, where $L_a$ is an additive permutation on $\F_q$ and $L_a(0)=0$;
	\item[(b)]  $\cL := \{l_{L_b(y),y}: y\in \F_q\} \cup \{l_0, l_{\alpha}\}$, where $\alpha\in \F_q^*$, $L_b$ is a two-to-one additive map on $\F_q$ and $L_b(0)=0$.
\end{enumerate}

Similar to Theorem \ref{th:classical_line_hyperovals}, we can also prove that for each hyperoval of type (b) in  $\Pi(\K_n)$, there is an associated line hyperoval of type (b) in  $\Pi(\K^{td}_n)$, or equivalently a translation hyperoval of type (b) in $\Pi(\K_n^t)$.

Assume that linearized polynomial $L$ and $\alpha\in\F_q$ together define a hyperoval $\Oval_L$ of type (b) in $\Pi(\K_n)$. More precisely, 
\[\Oval_L = \{(L(y),y): y \in \F_{q}\} \cup \{(0), (\alpha)\}.\]
It implies
\begin{equation}\label{eq:number_of_solution_Lya}
	\#\{y : L(y)*a+y=0\}= \left\{
	  \begin{array}{ll}
	    1, & \hbox{if $a=\alpha$;} \\
	    2, & \hbox{if $a\neq 0, \alpha$.}
	  \end{array}
	\right.
\end{equation}
Over the chosen trace-orthogonal basis $\mathscr{B}$ (see the discussions above \eqref{eq:commutative_symplectic_0}), we view $y$ as a vector in $\F_2^n$. As both maps $x\mapsto L(x)$ and $x\mapsto x*a$ are linear over $\F_2$, they defines two $n\times n$ matrices $N_L$ and $M_a$ over $\F_2$, respectively. It means that $L(y)*a+y$ can be written as
\[y(N_LM_a + I).\]
By \eqref{eq:number_of_solution_Lya}, we have
\begin{equation}
	\mathrm{rank}(N_LM_a + I)= \left\{
	 \begin{array}{ll}
	  0, & \hbox{if $a=\alpha$;} \\
	  1, & \hbox{if $a\neq 0, \alpha$.}
	 \end{array}
	 \right.
\end{equation}
On the other hand, as $M_a$ is always invertible for nonzero $a$, we have $N_LM_a + I= (N_L+ M_a^{-1}) M_a$. Thus
\[\mathrm{rank}(N_L+ M_a^{-1})=\mathrm{rank}(N_LM_a + I).\]
Together with 
$(N_L + M_a^{-1})^TM_a^T=N_L^TM_a^T+ I$, we have
\begin{equation}\label{eq:rank=rank}
	\mathrm{rank}(N_L^TM_a^T+ I)=\mathrm{rank}(N_LM_a + I).
\end{equation}
Translate $y(N_L^TM_a^T+ I)$ back into the notations of semifields, it becomes $a\circ L(y)+y$. By \eqref{eq:number_of_solution_Lya} and \eqref{eq:rank=rank}, 
\[	\#\{y : a\circ \bar{L}(y)+y=0\}= \left\{
	  \begin{array}{ll}
	    1, & \hbox{if $a=\alpha$;} \\
	    2, & \hbox{if $a\neq 0, \alpha$.}
	  \end{array}
	\right.
\]
Therefore, we have proved the following result.
\begin{theorem}\label{th:classical_line_hyperovals_b}
	Let $L$ be a linearized polynomial on $\F_q$ and $\alpha\in \F_q^*$. Then $L$ and $\alpha$ defines a hyperoval of type (b) in $\Pi(\K_n)$ if and only if the set of lines defined by
	\[\cL_c :=\{l_{\bar{L}(m),m} : m\in \F_q \}\cup 
			\{l_0,{l}_\alpha\},\] where $l_{\bar{L}(m),m}=\{(x,y): y=x\circ \bar{L}(m) + m\} \cup \{(m)\}$
	is a line hyperoval in $\Pi^{td}(\K_n)$.
\end{theorem}

By applying Theorem \ref{th:classical_line_hyperovals_b} onto the translation hyperoval in Theorem \ref{th:hyperoval_g1g2} and onto those contained in Table \ref{ta:hyperoval_b}, we get line hyperovals. 
\begin{corollary}\label{coro:classical_symplectic_b}
	Let $q=2^n$. In the symplectic presemifield planes $\Pi(\K_n^{td})$.
	\begin{itemize}
		\item	The set of lines defined by
			\[\cL_c :=\{l_{m+\sqrt{m},m} : m\in \F_q \}\cup 
				\{l_0, {l}_1\}\]
			is a line hyperoval.
		\item In particular, when $n=5$, there are exactly $12$ line hyperovals,  each of which contains $l_0$ and another line not equaling $l_\infty$.
	\end{itemize}
\end{corollary}
From the proof, it is obvious that Theorems \ref{th:classical_line_hyperovals} and \ref{th:classical_line_hyperovals_b} also hold for any translation hyperovals in an arbitrary semifield plane of characteristic $2$.

\section{Hyperovals in Symplectic planes}\label{se:symplectic_hyperoval}
In the symplectic presemifield plane $\Pi(\K_n^{td})$, we cannot find any translation hyperoval of type (a), which contains the points $(0)$ and $(\infty)$. However we still can construct translation hyperovals of type (b) with the points $(0)$ and $(\alpha)$ on $l_\infty$. First we give an infinite family of them.
\begin{theorem}\label{th:hyperoval_symplectic}
	Let $d$ and $n$ be relatively prime positive integers and $q:=2^n$. Then the following set
	\[\Oval_d := \{ (y^{2^d}+y, y): y\in \F_q \} \cup \{(0), (1)\} \]
	is a hyperoval in $\Pi(\K_n^{td})$.
\end{theorem}
\begin{proof}
	Again we only have to show that,  for every $m\in \F_q\setminus \{0\}$, there is at most one $y\in \F_q^*$ such that
	\[(y^{2^d}+y)\circ m + y=0,\]
	i.e.
	\begin{equation}\label{eq:hyperoval_symplectic_main}
		(y^{2^d}+y)m + \Tr((y^{2^{d+1}}+y^2)m) + y=0.
	\end{equation}
	
	When $m=1$, \eqref{eq:hyperoval_symplectic_main} becomes
	\[(y^{2^d}+y)+y=0,\]
	which has exactly one root $0$. This implies that the line $l_{1,0}$ of $\Pi(\K_n^{td})$ intersects $\Oval_d$ in the affine point $(0,0)$ and in $(1)$. 
	
	Instead of showing that  for every $m\in \F_q\setminus \{0,1\}$, there is exactly one $y\in \F_q^*$ satisfying \eqref{eq:hyperoval_symplectic_main}, we proceed to show that, for each $y\in \F_{q}\setminus \{0,1\}$, there is a unique $m\in \F_{q}\setminus \{0,1\}$ satisfying \eqref{eq:hyperoval_symplectic_main}. Denote this $m$ by $\mu(y)$. If we can further show that $\mu: \F_{q}\setminus \{0,1\}\rightarrow \F_{q}\setminus \{0,1\}$ is a permutation, then we finish the proof.
	
	For a given $y\in \F_{q}\setminus \{0,1\}$, if $\Tr((y^{2^{d+1}}+y^2)m)=0$, then from \eqref{eq:hyperoval_symplectic_main} we get
	\[m=\frac{1}{y^{2^d-1}+1};\]
	otherwise,
	\[m=\frac{y+1}{y^{2^d}+y}.\]
	However, in both cases, it is not difficult to see that 
	\[\Tr((y^{2^{d+1}}+y^2)m)=\Tr((y^{2^d} + y)\cdot y),\]
	which means that there is a unique $m$ such that \eqref{eq:hyperoval_symplectic_main} holds. In another word, 
	\[
	\mu(y) = \left\{
	  \begin{array}{ll}
	  \frac{y}{y^{2^d}+y},& \hbox{if $\Tr((y^{2^d}+y) \cdot y)=0$}; \\
	  \frac{y+1}{y^{2^d}+y},& \hbox{if $\Tr((y^{2^d}+y) \cdot y)=1$}.
	  \end{array}
	\right.
	\]
	
	Finally, we proceed to show that $\mu$ is a permutation on $\F_{q}\setminus \{0,1\}$. Assume that there exists $y,z\in \F_q\setminus{0,1}$ such that $\mu(y)=\mu(z)$. Without loss of generality, we assume that $\Tr((y^{2^d}+y) \cdot y)=0$. i.e.\ $\mu(y)=\frac{y}{y^{2^d}+y}$; for those $y$ satisfying $\mu(y)=\frac{y+1}{y^{2^d}+y}$, we can get a result similar to the rest of this proof.
	
	If $\mu(z)=\frac{z}{z^{2^d}+z}$, then
	\[\frac{z}{z^{2^d}+z}=\frac{y}{y^{2^d}+y},\]
	from which we can derive that $y^{2^d-1}=z^{2^d-1}$. It means that $y=z$ due to $\gcd(2^d-1,2^n-1)=1$.
	
	If $\mu(z)=\frac{z+1}{z^{2^d}+z}$, i.e.\ $\Tr((z^{2^d}+z) \cdot z)=1$, then we take $w:= z+1$. That means
	\[\mu(z)=\frac{w}{w^{2^d}+w}=\frac{y}{y^{2^d}+y}.\]
	Again, according to $\gcd(2^d-1,2^n-1)=1$, we get $y=w=z+1$. It implies that
	\[\Tr((y^{2^d}+y) \cdot y)=\Tr((w^{2^d}+w) \cdot w)=\Tr((z^{2^d}+z) \cdot z + (z^{2^d}+z))=1,\]
	which is a contradiction to $\Tr((y^{2^d}+y) \cdot y)=0$.
\end{proof}

\begin{theorem}
	Let $n$, $d$ and $e$ be in $\Z^+$ such that $\gcd(n,d)=\gcd(n,e)=1$, and $q=2^n$. Then the hyperovals $\Oval_d$ and $\Oval_e\in \Pi(\K_n^{td})$ are equivalent if and only if $d\equiv \pm e\pmod{n}$.
\end{theorem}
\begin{proof}
	Assume that hyperoval $\Oval_d$ and $\Oval_e$ are equivalent. By definition, there exists $\varphi\in \Aut(\Pi(\K_n^{td}))$ such that $\varphi(\Oval_d)=\Oval_e$. By Theorem \ref{th:automorphism_semifieldplane}, there exists a translation $\tau$, a shear $\eta$ and $\gamma\in \Aut(\bbS)$, such that $\varphi(P)=\gamma (\eta (\tau(P)))$ for all point $P\in \Pi(\K_n^{td})$. Hence, we can write
	\begin{equation}\label{eq:equivalence_last}
		\eta(\tau(\Oval_d))=\gamma^{-1} (\Oval_e). 
	\end{equation}
	
	Noting that the autotopism group $\Aut(\K_n^{td})$ is the same as $\Aut(\K_n)$, the point $(y^{2^e}+y, y)$ in $\Oval_e$ is mapped to $(y^{2^{e+1}}+y^2, y^2)$ and the point $(1)$ is fixed under $\gamma^{-1}$. Hence $\gamma^{-1}(\Oval_e)=\Oval_e$.
	
	As $\Oval_d\cap l_\infty=\Oval_e\cap l_\infty=\{(0), (1)\}$, $\eta$ must be either the identity map or it interchanges the points $(0)$ and $(1)$. 
	
	Suppose that $\tau(x,y) := (x+a,y+b)$. If $\eta$ is the identity map, then from \eqref{eq:equivalence_last} we have
	\[(y^{2^d}+y+a, y+b)=(z^{2^e} + z, z),\]
	which means $(y+b)^{2^e} + (y+b)=y^{2^d}+y+a$ for all $y\in \F_q$. It holds if and only if $e=d$ and $b^{2^e}+b=a$.
	
	If $\eta$ interchanges the points $(0)$ and $(1)$, then from \eqref{eq:equivalence_last} we have
		\[(y^{2^d}+y+a, y+b + (y^{2^d}+y+a)\circ 1)=(z^{2^e} + z, z),\]
	i.e.
	\[(y^{2^d}+y+a, y^{2^d}+a+b)=(z^{2^e} + z, z).\]
	It is equivalent to 
	\[ (y^{2^d}+a+b)^{2^e} + y^{2^d}+a+b = y^{2^d}+y+a,\]
	for all $y$. It holds if and only if $e\equiv -d \pmod{n}$ and $b^{2^d}+b=a$.	
\end{proof}

By using MAGMA program, we show that there is no translation hyperoval of type (a) and $10$ inequivalent translation hyperovals of type (b) in $\Pi(\K^{td}_5)$. The corresponding functions $f_b$ and the points $(\alpha)$ are listed in Table \ref{ta:hyperoval_c}.

\begin{table}[h!]
\caption{Functions defining translation hyperovals of type (b) in $\Pi(\K_5^{td})$}
\label{ta:hyperoval_c}
\centering{
	\begin{tabular}{|c|c|l|}
	\hline No. &  $\alpha$ &  {Functions}  \\
	\hline
	$1$ & $1$ 	 & $y^2 + y$, \Tstrut \\
	$2$ &		 & $y^4 + y$  \Tstrut\\ \hline
	$3$ &$\omega$ & $\omega^{24} y^{16} + \omega^{18} y^8 + \omega y^4 + \omega^{29} y^2 + \omega^{15} y$, \Tstrut \\
	$4$ &		 & $\omega^{17} y^{16} + \omega^{12} y^8 + \omega^{17} y^4 + \omega^5 y^2 + \omega^{13} y$, \Tstrut\\ 
	$5$ &	     & $\omega^7 y^{16} + \omega^{24} y^8 + \omega^4 y^4 + \omega^{14} y^2 + \omega^5 y$, \Tstrut\\
	$6$ &	     & $\omega^{28} y^{16} + \omega^7 y^8 + \omega^{23} y^4 + \omega^{27} y^2 + y$ \Tstrut\\ \hline	
	
	$7$ & $\omega^3$ & $\omega^5 y^{16} + \omega^5 y^8 + y^4 + \omega^{28} y^2 + \omega^{24} y$ \Tstrut \\ \hline
	
	$8$ & $\omega^5$ & $\omega^{19} y^{16} + \omega^{27} y^8 + \omega^{26} y^4 + \omega^{25} y^2 + \omega^{22} y$ \Tstrut \\ \hline
	
	$9$ & $\omega^{15}$ & $\omega^5 y^{16} + \omega y^8 + \omega^{21} y^4 + \omega^{29} y^2 + \omega^{18} y$, \Tstrut \\
	$10$ &			  &	$\omega^{30} y^{16} + \omega^{13} y^8 + \omega^4 y^4 + \omega^{10} y^2 + \omega^{18} y$  \Tstrut\\	\hline
	\multicolumn{3}{|c|}{\small Here the primitive element $\omega$ satisfies $\omega^5 + \omega^2 + 1=0$.}	\Tstrut\\\hline 
	\end{tabular}
}
\end{table}

\section{Designs associated with hyperovals}\label{se:designs}
Let $\Pi$ be a projective plane of even order $q$ and $\Oval$ a hyperoval in it. We can derive several different type of designs from $\Oval$. For instance, Maschietti in \cite{maschietti_hyperovals_1992} showed that, for each hyperoval,  there is always an associated Hadamard $2$-design with parameters $(q^2-1, q^2/2, q^2/4)$. When the hyperoval is defined by a monomial and the plane is desarguesian, Maschietti constructed an associated $(q-1,q/2-1, q/4-1)$-difference set in \cite{maschietti_difference_1998}; see \cite{dillon_new_2004,maschietti_hadamard_1995} too. A $(v,k,\lambda)$-difference set is a $k$ subset of a group $G$ of order $v$, such that each nonidentity element in $G$ can be written as $d_1d_2^{-1}$ with $d_1,d_2\in D$ in exactly $\lambda$ ways; see \cite{beth_design_1999,pott_finite_1995} for more details.

In this section, for an arbitrary hyperoval of type (a) in a semifield plane, we consider another associated symmetric design with parameter $(q^2, {q^2}/{2} + {q}/{2},{q^2}/{4} + {q}/{2})$ and the corresponding difference sets in two different abelian groups.

\begin{lemma}\label{lm:intersection_ovals}
	Let $\Oval_1$ and $\Oval_2$ be two hyperovals in a projective plane $\Pi$. Assume that $\Oval_1$ and $\Oval_2$ have exactly two common points $P$ and $Q$. Let $\rS(\Oval_1)$ and $\rS(\Oval_2)$ be the secant lines of $\Oval_1$ and $\Oval_2$ respectively. Let $\mathscr{P}(P)$ be the pencil of lines through point $P$. Then
	\begin{equation}\label{eq:cardinality_secant_lines}
		\#(\rS(\Oval_1)\setminus\mathscr{P}(P)) = \frac{q(q+1)}{2};
	\end{equation}
	\begin{equation}\label{eq:cardinality_secant_lines_two_ovals}
		\#((\rS(\Oval_1)\cap \rS(\Oval_2)) \setminus\mathscr{P}(P)) = \frac{q^2}{4} + \frac{q}{2}.
	\end{equation}
\end{lemma}
\begin{proof}
	For any hyperoval $\Oval$ in $\Pi$, there are totally $\frac{(q+1)(q+2)}{2}$ secant lines and $q+1$ ones through a given point on $\Oval$. Hence there are
	$\frac{(q+1)(q+2)}{2}-(q+1)$ secant lines in $\rS(\Oval_1)\setminus\mathscr{P}(P)$.
	
	Let $R$ be a point in $\Oval_2\setminus \{P,Q\}$. Then there are exactly $\frac{q+2}{2}$ lines containing $R$ secant to $\Oval_1$. Among these secant lines, there is one through $P$ and one through $Q$. Furthermore, as every secant line to $\Oval_1$ passes through exactly two points in $\Oval_2$, there are totally $\frac{q}{2} \left(\frac{q}{2}-1\right)$ lines secant to $\Oval_1$ and $\Oval_2$ simultaneously, which do not contain $P$ and $Q$. Together with the $q$ secant lines through $Q$, we get \eqref{eq:cardinality_secant_lines_two_ovals}.
\end{proof}



Now we consider the action of collineations on translation hyperovals in $\Pi(\K_n)$.

\begin{lemma}\label{lm:q2_hyperovals}
	Let $\bbS$ be a presemifield of order $q$ with multiplication $\star$. Let $\Oval$ be a translation hyperoval of type (a) in $\Pi(\bbS)$. Under the collineation group $\Sigma$ of $\Pi(\bbS)$ (see Theorem \ref{th:automorphism_semifieldplane}), $\Oval$ is mapped to $q^2$ hyperovals, every two of which meets in exactly two points, one of which is $(\infty)$.
\end{lemma}
\begin{proof}
	Suppose that $\Oval=\{(x,f(x)): x\in \F_q\} \cup \{(0), (\infty)\}$, where $f$ is additive over $\F_q$. Recall that the collineation group $\Sigma=\autfix{l_\infty}{l_\infty}\rtimes \autfix{(\infty)}{l_0}$, each element $\tau_{a,b}\in \autfix{l_\infty}{l_\infty}$ are defined by 
	\begin{align*}
		\tau_{a,b}:~~& (x,y)\mapsto (x+a,y+b),\\
					& (z) \mapsto (z),
	\end{align*}
	for $x,y\in \F_q$ and $z\in \F_q\cup \{\infty\}$, and $\sigma_c\in \autfix{(\infty)}{l_0}$ is defined by
	\begin{align*}
		\sigma_c:~~& (x,y)\mapsto (x,y+x\star c),\\
					& (z) \mapsto (z+c),\\
					& (\infty) \mapsto (\infty),
	\end{align*}
	for $x,y,z\in \F_q$. It is clear that $\#\Sigma=q^3$.
	
	Under the element $\sigma_c\tau_{a,b}\in\Sigma$, $\Oval$ is mapped to
	\[ \{(x+a, f(x)+(x+a)\star c+b) : x\in \F_q  \}\cup \{(c), (\infty)\}, \]
	i.e.
	\[ \{(x, f(x)+x\star c + f(a) +b ) : x\in \F_q  \}\cup \{(c), (\infty)\},\]
	which we denote by $\Oval'$. It is straightforward to see that there are totally $q^2$ distinct hyperovals $\Oval'$ for all $a,b,c\in \F_q$.
	
	To consider the intersection of any two hyperovals which are obtained by applying $\Sigma$ on $\Oval$, we only have to look at the intersection of $\Oval$ and $\Oval '$, where $\Oval\neq \Oval'$. It is clear that $(\infty)$ is always in $\Oval\cap \Oval'$.
	
	First we suppose that $c\neq 0$. Considering 
	\[(x,f(x))=(x, f(x)+x\star c + f(a) +b ),\]
	we see that there is exactly affine point on $\Oval$ and $\Oval'$ simultaneously, because $x\star c + f(a) +b=0$ has a unique solution in $x$.
	
	Second we suppose that $c = 0$. As $\Oval\neq  \Oval'$, we have $f(a) +b\neq 0$ and $\Oval' \cap \Oval =\{(0), (\infty)\}$.
	
	Therefore, when $\Oval$ and $\Oval'$ are different, there are exactly two points in $\Oval\cap \Oval'$, one of which is $(\infty)$.
\end{proof}

Let us use $\U(\Oval)$ to denote the set $q^2$ hyperovals obtained in Lemma \ref{lm:q2_hyperovals}. It is readily verified that in $\Sigma$ there are two subgroups acting transitively on $\U(\Oval)$. They are
	\[G_1 := \{ \tau_{0,b}\sigma_c: b,c\in \F_q \},\]
and 
	\[G_2 := \{ \tau_{a,b}\sigma_{a}: a,b\in \F_q \}.\]
It is routine to verify that $G_1$ is isomorphic to $C_2^{2n}$ and $G_2$ is isomorphic to $C_4^n$, where $C_k$ denotes the cyclic group of order $k$.

Now, let $\bbS$ be a (pre)semifield of order $q$. Let $\cV$ denote the set of lines in $\Pi(\bbS)$ not through $(\infty)$. Assume that there is a translation hyperoval $\Oval$ of type (a) in $\Pi(\bbS)$. Let $\cB$ be defined by
	\[ \cB := \{\rS(\Oval')\setminus\mathscr{P}(\infty)  : \Oval'\in \U(\Oval)\}.  \]

Finally, let $\cD(\Oval)$ be the incidence structure whose point set is $\cV$ and whose blocks are in $\cB$, and the incidence relation is set membership. We may prove the following result.

\begin{theorem}\label{th:design_1}
	The incidence structure $\cD(\Oval)$ is a symmetric design with parameters 
	\[\left(q^2, \frac{q^2}{2} + \frac{q}{2}, \frac{q^2}{4} + \frac{q}{2}\right).\] 
	Furthermore, there are two Hadamard difference sets associated with $\Oval$ in $G_1$ and $G_2$, respectively.
\end{theorem}
\begin{proof}
	It is easy to see that there are totally $q^2$ lines in $\cV$. Taking $P=(\infty)$ in Lemma \ref{lm:intersection_ovals}, we see that every block in $\cB$ is of size $\frac{q^2}{2} + \frac{q}{2}$ and the intersection of any two blocks is of size $\frac{q^2}{4} + \frac{q}{2}$. Therefore, we obtain a $\left(q^2, \frac{q^2}{2} + \frac{q}{2}, \frac{q^2}{4} + \frac{q}{2}\right)$-design which is symmetric.
	
	It is not difficult to show that $G_1$ and $G_2$ are both sharply transitive on the lines in $\cV$. For $i=1,2$, if we take $D_i$ to be the $ \frac{q^2}{2} + \frac{q}{2}$ elements in $G_i$ which map $l_{0,0}$ to a line in $\rS(\Oval)\setminus\mathscr{P}(\infty)$, then $D_i$ is a difference set in $G_i$. Furthermore, the complement of $D_i$ is a $\left(q^2, \frac{q^2}{2} - \frac{q}{2}, \frac{q^2}{4} - \frac{q}{2}\right)$-difference set, namely a Hadamard difference set.
\end{proof}
\begin{remark}
	In \cite{kantor_symplectic_1975}, Kantor first obtained the difference set in $G_1$ as well as the design in Theorem \ref{th:design_1} in terms of line hyperovals in translation planes. In \cite{de_resmini_arcs_2002}, de Resmini, Ghinelli and Jungnickel proved the Hadamard difference set in $G_2$ for the special case $\Oval=\Oval_s$ by looking at the maximal arcs derived from $\Oval_s$.
\end{remark}

Two designs $\cD_1$ and $\cD_2$ are \emph{isomorphic} if there is a one-to-one map $\varphi$ from the points (blocks) of $\cD_1$ to the points (blocks) of $\cD_2$ such that a point $P$ is incident a block $l$ if and only if $\varphi(P)$ is incident with $\varphi(l)$. 
By using MAGMA program, we can show that the designs associated with the $5$ hyperovals of type (a) in Table \ref{ta:hyperoval_a} are pairwise non-isomorphic.

\begin{remark}
	In  \cite{cesmelioglu_bent_2015}, \c{C}e\c{s}melio\u{g}lu, Meidl and Pott showed that every hyperoval (not necessarily translation ones), which contains the point $(\infty)$, in a translation plane can be used to construct a bent function, which generalizes a construction by Dillon in \cite{dillon_elementary_1974}; see \cite{carlet_more_2015,carlet_dillons_2011} too. Also, in Section $4$ of \cite{cesmelioglu_bent_2015}, searching for more $o$-polynomials for (pre)semifields or, in geometric terms, for more hyperovals of type $(a)$ in (pre)semifield planes, is posed as an open problem. In Sections \ref{se:commutative}, \ref{se:symplectic} and \ref{se:symplectic_hyperoval}, we found some new examples of such objects. In particular, we can derive boolean bent functions from the standard hyperovals $\Oval_s$, the hyperovals in Table \ref{ta:hyperoval_a} and Corollary \ref{coro:classical_symplectic}. A Boolean bent function on $\F_{2}^{2n}$ is equivalent to a Hadamard difference set (or its complement) in the elementary abelian $2$-group $C_2^{2n}$. It implies that, from one hyperoval containing $(\infty)$, we can derive another symmetric design with the same parameters as those obtained in Theorem \ref{th:design_1}. It is remarkable to point out that these two associated designs with the same hyperoval are not necessarily isomorphic. For instance, when the hyperoval is defined by $f(x)=x^2$ in $\Pi(\K_5)$, our MAGMA program shows that the associated two designs are not isomorphic.
\end{remark}

It is a natural question to ask whether we can also derive similar difference sets and designs from translation hyperplanes of type (b) in (pre)semifield planes. Unfortunately, the following result shows that Lemma \ref{lm:q2_hyperovals} does not hold for them, which means that we cannot follow the construction in Theorem \ref{th:design_1} to get designs or difference sets.
\begin{theorem}\label{th:q2_hyperovals_b}
	Let $\bbS$ be a presemifield of order $q$ with multiplication $\star$. Let $\Oval$ be a translation hyperoval of type (b) defined by a function $f$ over $\F_q$ and the points on $l_\infty$ are $(0)$ and $(\alpha)$. Let $\theta$ be the unique element in $\F_q$ such that $f(\theta)=0$. Under the collineation group $\Sigma$ of $\Pi(\bbS)$ (see Theorem \ref{th:automorphism_semifieldplane}), $\Oval$ is mapped to $q^2$ hyperovals. 
	
	If there exists $v\in \F_q$ such that $f(v)\star \alpha =\theta$, then every two such hyperovals meet in $0$, $2$, $4$ or $6$ points; otherwise every two such hyperovals meet in $0$, $2$ or $4$ points.
\end{theorem}
\begin{proof}
	By assumption, our hyperoval is 
	\[\Oval=\{(f(x),x): x\in \F_q\} \cup \{(0), (\alpha)\},\] 
	where $\alpha \in \F_q^*$ and $f$ is two-to-one and additive over $\F_q$. Let $\beta$ be  an element not in the image of $f$. Then $\F_q= f(\F_q) \dot{\cup} (\beta + f(\F_q))$.
	
	Under an element $\sigma_c\tau_{a,b}\in\Sigma$, $\Oval$ is mapped to
	\[ \{f(x)+a, x+(f(x)+a)\star c+b) : x\in \F_q  \}\cup \{(c), (\alpha+c)\}, \]
	which we denote by $\Oval'$.
	
	Depending on the value of $a$, we consider $\Oval'$ in the following two cases:
	\begin{itemize}
		\item If $a\in f(\F_q)$, then there exists $u$ such that  $a=f(u)=f(u+\theta)$ and
		\begin{align}
		\nonumber	\Oval'&=\{ (f(x+u),  (x+u) + f(x+u)\star c + (b+u)) : x\in \F_q \}\cup \{ (c), (\alpha+c) \}\\
		\nonumber		  &=\{ (f(y),  y + f(y)\star c + (b+u)) : y\in \F_q \}\cup \{ (c), (\alpha+c) \}\\
		\label{eq:oval_orbit_1}		  &=\{ (f(y),  y + f(y)\star c + (b+u+\theta)) : y\in \F_q \}\cup \{ (c), (\alpha+c) \}.
		\end{align}
		There are exactly $q^2/2$ hyperovals in the form of \eqref{eq:oval_orbit_1}.
		\item If $a\notin f(\F_q)$, then there exists $u$ such that  $a=f(u)+\beta=f(u+\theta)+\beta$ and
		\begin{align}
		\nonumber	\Oval' &=\{ (f(y)+\beta,  y + f(y)\star c + (b+u)) : y\in \F_q \}\cup \{ (c), (\alpha+c) \}\\
		\label{eq:oval_orbit_2}		  &=\{ (f(y)+\beta,  y + f(y)\star c + (b+u+\theta)) : y\in \F_q \}\cup \{ (c), (\alpha+c) \}.
		\end{align}
		There are exactly $q^2/2$ hyperovals in the form of \eqref{eq:oval_orbit_2}.
	\end{itemize}
	
	To consider the intersection of any two hyperovals which are obtained by applying $\Sigma$ on $\Oval$, we only have to look at the intersection of $\Oval$ and $\Oval '$, where $\Oval\neq \Oval'$.
	
	By comparing $\Oval$ and \eqref{eq:oval_orbit_2}, we have
		\[
			\#(\Oval \cap \Oval' )= \left\{
			  \begin{array}{ll}
				   2, & \hbox{if $c=0$ or $\alpha$;} \\
				   0, & \hbox{otherwise.}
			  \end{array}
			\right.
		\]
	
	Similarly, comparing $\Oval$ and \eqref{eq:oval_orbit_1}, we see that
	\[
		\#(\Oval \cap \Oval' )= \left\{
		  \begin{array}{ll}
		    2\cdot \#(\{z: z\star c=b+u \text{ or } b+u+\theta \}\cap f(\F_q))+2, & \hbox{if $c=0$ or $\alpha$;} \\
		    2\cdot \#(\{z: z\star c=b+u \text{ or } b+u+\theta \}\cap f(\F_q)), & \hbox{otherwise.}
		  \end{array}
		\right.
	\]
	Moreover,
	\begin{align*}
		&\#(\{z: z\star c=b+u \text{ or } b+u+\theta \}\cap f(\F_q))\\ 
		= &\left\{
			   \begin{array}{ll}
			     1, & \hbox{if $\{y: f(y)\star c=\theta \}=\emptyset$;} \\
			      2\cdot \#(\{z: z\star c=b+u\}\cap f(\F_q)), & \hbox{otherwise.}
			   \end{array}
			 \right.
	\end{align*}	

	As $z \mapsto z\star c$ is a permutation when $c\neq 0$, we get that $\#(\Oval \cap \Oval' )\in \{0,2,4,6\}$. It is not difficult to see that, for $k\in\{0,2,4\}$, by choosing appropriate $c$ and $b+u$, we can always find a hyperoval $\Oval'$ intersecting $\Oval$ at $k$ points.
	
	In particular, $\#(\Oval \cap \Oval' )=6$ if and only if $c=\alpha$ and there exist $v,w\in \F_q$ such that 
	\[f(v)\star \alpha =\theta \quad \text{and}\quad f(w)\star \alpha =b+u. \qedhere\]
\end{proof}

It is not difficult to verify that there does not exist $v\in \F_{2^n}$ such that $(v^2+v)* 1 =1$ or $(v^{2^d}+v) \circ 1 =1$, where $\gcd(d,n)=1$. It means that, if $\Oval$ is the hyperoval $\Oval_g$ in $\Pi(\K_n)$ defined in Theorem \ref{th:hyperoval_g1g2} or $\Oval_d$ in $\Pi(\K_n^{td})$ given in Theorem \ref{th:hyperoval_symplectic}, then $\#(\Oval \cap \Oval')\neq 6$ for all the $q^2-1$ hyperovals $\Oval'$. 

Furthermore, we check all the sporadic hyperovals listed in Table \ref{ta:hyperoval_b} and Table \ref{ta:hyperoval_c} by MAGMA program. If $\Oval$ is one of the hyperovals of No.\ $3$, $8$, $10$, $12$ in Table \ref{ta:hyperoval_b} and No.\ $3$, $6$, $9$, $10$ in Table \ref{ta:hyperoval_c}, then there exists $\Oval'$ such that $\#(\Oval\cap\Oval')=6$; otherwise $\#(\Oval\cap\Oval')\in \{0,2,4\}$.

\section*{acknowledgment}
This work is supported by the Research Project of MIUR (Italian Office for University and Research) ``Strutture geometriche, Combinatoria e loro Applicazioni" 2012. The third author is supported by the Alexander von Humboldt Foundation.

\end{document}